\documentclass[11pt,twoside]{article}

\usepackage{graphicx}
\usepackage{rotating}
\usepackage{amsmath}
\usepackage{amsfonts}
\usepackage{amssymb}
\usepackage{amsthm}
\usepackage{tabularx}
\usepackage[all]{xy}

\newtheorem{thm}{Theorem}[section]
\newtheorem{pro}[thm]{Proposition}
 \newtheorem{cor}[thm]{Corollary}
 \newtheorem{exa}[thm]{Example}
 
 \newtheorem{lem}[thm]{Lemma}
  \newtheorem{defn}[thm]{Definition}

\begin{document}
%===============

%\bibliographystyle{UKM} % UKM.bst

 \centerline{\large{\textbf{SOME APPLICATIONS FOR FADELL-DOLD THEORM     }}}
\centerline{\large{\textbf{ IN FIBRATION THEORY BY USING  }}}
\centerline{\large{\textbf{   HOMOTOPY GROUPS }}} \centerline{}

\centerline{{\emph{\textbf{Amin  Saif and Adem K{\i}l{\i}\c cman}}}}
\centerline{\footnotesize{Institute of Mathematical Research
(INSPEM), University Putra Malaysia}}
\centerline{\footnotesize{43400 UPM, Serdang, Selangor,
Malaysia}}

%%%-------------------------------------------------

%%%----------------------------------------------------

\begin{abstract}
\footnotesize{ The purpose of this paper is to give some solutions
for the classification problem in fibration theory by using the
homotopy sequences of fibrations (sequences of $n$-th homotopy
groups $ \pi_{n}(S,s_{o}) $ of total spaces of fibrations). In
particular, to show the role of homotopy sequence of $n$-th
homotopy to get the required fiber map in Fadell-Dold theorem such
that the restriction of this fiber map on some fiber spaces is a
homotopy equivalence.\\
\quad\\
 \emph{Keywords}:  Fibration;
    homotopy group;   homotopy equivalence.\\
\quad\\
 \emph{ AMS classification}: 55P10, 55R55,  55Q05}

\end{abstract}

%%% #############################################################################

\section{Introduction}

\noindent The homotopy theory of topological spaces attempts to
classify weak homotopy types of spaces and homotopy classes of
maps. The classification of maps within a homotopy is a central
problem in topology and several authors contributed in this area,
see for example the related works in \cite{kirby}.\\

\noindent The concepts of Hurewicz fibrations have played very
important roles for investigating the mutual relations of among
the objects. For this purpose Coram and Duvall \cite{coram2}
introduced an approximate fibration as a map having the
approximate homotopy lifting property for every space, which is a
generalization of a Hurewicz fibration having valuable properties
similar to the Hurewicz fibration and is widely applicable to the
maps whose fibers are nontrivial shapes. Thus it is very essential
to examine whether a given decomposition map is an approximate
fibration, for exact homotopy sequence that will provide us
structural informations about any one object by means of their
interrelations with the others, Coram and Duvall \cite{coram1}
gave several characterizations for an approximate fibration. \\

\noindent In \cite{dwyer1}, Dwyer and Kan followed the simplicial
model category and introduced weak equivalences between the
objects. Further, Dwyer and Kan in \cite{dwyer4} define a notion
of equivalence of simplicial localizations by using simplicial
sets for the diagrams, which provide an answer to the question
that posed by Quillen on the equivalence of homotopy theories in
\cite{quillen}. In fact, the category of simplicial localizations
together with this notion of equivalence gives rise to a "homotopy
theory of homotopy
theory", see \cite{dwyer}.\\

\noindent There has also been some further developments leading to
classification in the homotopy type of newer manifolds(Wall's
manifolds, Milnor manifolds, etc.) which form generators of
several different groups of manifolds, see more details
\cite{kirby}, \cite{charles} and \cite{wall} by using the $\Pi$-algebras , see \cite{blanca}.\\

\noindent Recall the problems of classifying Hurewicz fibrations
whose fibres have just two non-zero homotopy groups which are very
interesting study in homotopy theory. In what follows, $S^I$ will
denote the path apace ( with the compact-open topology) of a space
(Hausdorff space) $S$, $ \Omega(S,s_{o})$ will denote the loop
space in $S^I$ based a point $s_{o}$, $\widetilde{s}$ a constant
path  into $s\in S$, $\overline{\alpha}$ the inverse path of
$\alpha\in S^I$, $\star$ the usual path multiplication operation
and $ \simeq  $ the same
homotopy type for spaces and homotopic for maps.\\

\noindent Let $f:S\longrightarrow O$ be a fibration with a base
$O$, total space $S$ and fiber space $F_{r_{o}}=f^{-1}(r_{o})$,
where $r_{o}\in O$. A map $L_{f}: \bigtriangleup f \longrightarrow
S^I$ is called a \emph{ lifting function} for $f$ if
$L_{f}(s,\alpha)(0)=s$ and $f[L_{f}(s,\alpha)]=\alpha$ for all
$(s,\alpha)\in \bigtriangleup f$, where  $\bigtriangleup f =
\{(s,\alpha)\in S\times O^I:f(s)=\alpha(0)\}$. If $L_{f}(s,f\circ
\widetilde{s})=\widetilde{s} \mbox { } \mbox{ for all } \mbox{
}s\in S$, then   the  lifting function is called a \emph{  regular
lifting function}. A  fibration $f$ is called \emph{regular
fibration} if it has regular lifting function.\\

\noindent Recall the Curtis-Hurewicz theorem, \cite{Hur1}, which
is one of the famous theorems in fibration theory which shows that
any map is regular fibration if and only if it has regular lifting
function. One of the main problems in fibration theory is a
classification problem which is given by:\newline Under what
conditions two fibrations, over a common base, will be fiber
homotopy equivalent? Fadell-Dold theorem, \cite{Fad1}, is one of
the solutions of this problem which clarifies  that if the common
base $O$ of two fibration $f_{1}:S_{1}\longrightarrow O$ and
$f_{2}:S_{2}\longrightarrow O$ is a pathwise connected and an
absolute neighborhood retract (ANR), then $f_{1}$ and $f_{2}$ are
fiber homotopy equivalent if and only if there is a fiber map
$h:S_{1}\longrightarrow S_{2}$ such that the restriction map of
$h$ on $f^{-1}_{1}(r_{o})$ is homotopy equivalence into $
f^{-1}_{2}(r_{o})$, for some $r_{o}\in O$.\\

\noindent In general it is difficult to find the required fiber
map of Fadell-Dold theorem in Hurewicz fibration theory. Thus in
this paper, we show the role of homotopy sequence of $n-$th
homotopy groups $ \pi_{n}(S,s_{o}) $ of two fibrations  to get
this  required fiber map in Fadell-Dold theorem such that the
restriction of this fiber map on some fiber spaces is a   homotopy
equivalence.  That is, we give  some  solutions for the
classification problem by using the $n-$th homotopy groups $
\pi_{n}(S,s_{o}) $ of  total spaces of fibrations. \\

%%%%################################################################################
\section{Preliminaries }

\noindent For the $n-$th homotopy groups $\pi_{n}(S,s_{o})$ and
$n-$th relative homotopy groups $\pi_{n}(S,A,s_{o})$, recall
\cite{Spanier} that:
\begin{enumerate}
    \item If $h:S\longrightarrow O$ is a
    map then for a positive integer $ n> 0 $, there is a
    homomorphism     $\widehat{h}: \pi_{n}(S,s_{o})\longrightarrow
    \pi_{n}(O,h(s_{o}))$  defined by $\widehat{h}([\alpha])=[h\circ \alpha]$ . At
    $n=0$,      $\widehat{h}$ sends the path-components of $S$ into those of
     $O$. $\widehat{h}$ is called a homomorphism induced by  $h$.
    \item For a positive integer $ n> 1 $, there is a homomorphism
    (called boundary operator)    $\partial: \pi_{n}(S,A,s_{o})\longrightarrow
    \pi_{n-1}(A,s_{o})$  defined by $\partial([\alpha])=[\alpha|_{I^{n-1}\times\{0\}}]$.
     At $n=1$,
    $\alpha(I^{n})$ is a point in $A$
    which determines a path-component $C \in \pi_{0}(S,s_{o})$
    and $\partial([\alpha])=C$.
    \item $\pi_{n}(S,s_{o})$ is isomorphic to
    $\pi_{n-1}( \Omega(S,s_{o}),\widetilde{s_{o}})$, for a positive integer $n>0$.
    \item $\pi_{n}(S,s_{o})$ is abelian group for a positive integer $n>1$.
\end{enumerate}

%%====================================================================================

\begin{thm}\label{a111}
\emph{ \cite{Spanier} Let  $h:(S,s_o)\longrightarrow (O,h(s_o))$
be  a homotopy equivalence. Then the induced homomorphisms
$\widehat{h}:\pi_{n}(S,s_{o})\longrightarrow \pi_{n}(O,h(s_{o}))$
are isomorphisms  for
   a positive integer $ n> 0 $.}
\end{thm}

%%====================================================================================

\noindent Recall \cite{Spanier} that if $h$ is a fibration, then
Theorem \ref{a111} remains valid. The following theorem is the
consequences of Whitehead's \cite{Whiteh} and Hurewicz's theorems,
see \cite{bredon}.

\begin{thm}\label{a112}
\emph{   Let $S$ and $O$ be simply connected spaces which are
dominated by ANR's. If there is a map $f:S\longrightarrow O$
induces isomorphism between the $n-$th homotopy groups of $S$ and
$O$, then $f$ is   homotopy equivalence.}
\end{thm}

\noindent Basically, Whitehead's Theorem says that for
CW-complexes, if a map $f : X \to Y$ induces an isomorphism on all
homotopy groups then it is a homotopy equivalence. But, as the
example above shows, you need the map. Such a map is called a weak
homotopy equivalence. We note that Whitehead's Theorem is not true
for spaces wilder than CW-complexes for example, the Warsaw circle
has all of its homotopy groups trivial but the unique map to a
point is not a homotopy equivalence.

%%====================================================================================
\begin{thm}\label{a113}
\emph{ \cite{Casson} Let $f:S\longrightarrow O$ be a   fibration
and  a fiber space $F_{r_{o}}$ be a pathwise connected ANR for
some $r_{o}\in O$. If $ \Omega(O,r_{o})\simeq ANR$, then $
\Omega(S,s_{o})$ is dominated by ANR for any $s_{o}\in F_{r_{o}}$.
If $S$ is a simply connected then $ \Omega(S,s_{o})$ is of the
same homotopy type with ANR space.}
\end{thm}

%%====================================================================================

\noindent The definition of homotopy sequence of a fibration
$f:S\longrightarrow O$ is given as follow:\newline Let $s_{o}\in
F_{r_{o}}$. We  can consider $f$ as a map of   a triple
$(S,F_{r_{o}},s_{o})$ into a pair $(O,r_{o})$. Then the following
sequence   is called a \emph{homotopy  sequence} of a fibration
$f$:
\begin{eqnarray*}
\xymatrix{...\pi_{n}(S,s_{o})\ar@{->}[r]^{\widehat{f}}
&\pi_{n}(O,r_{o})\ar@{->}[r]^{\partial_{\bullet}\quad}&\pi_{n-1}(F_{r_{o}},s_{o})\ar@{->}[r]^{\widehat{i}}&\pi_{n-1}(S,s_{o})...}\\
\xymatrix{...\pi_{2}(O,r_{o})\ar@{->}[r]^{\partial_{\bullet}}
&\pi_{1}(F_{r_{o}},s_{o})\ar@{->}[r]^{\widehat{i}}&\pi_{1}(S,s_{o})\ar@{->}[r]^{\widehat{f}}&\pi_{1}(O,r_{o})}\\
\xymatrix{\ar@{->}[r]^{\partial_{\bullet}\qquad}&\pi_{0}(F_{r_{o}},s_{o})\ar@{->}[r]^{\widehat{i}}&\pi_{0}(S,s_{o}),}
\end{eqnarray*}
where $i,j$ are inclusion maps and
$\partial_{\bullet}=\partial\circ (\widehat{f})^{-1}$.
 Recall \cite{Hur2} that this sequence is exact, that is, the kernel of each
homomorphism is equal to the image of the previous one.

%%====================================================================================
\begin{thm}\label{a114}
\emph{   \cite{Kees}  Let $f:S\longrightarrow O$ be a fibration
with pathwise connected space $O$. Then $ f^{-1}(r_{1})$ and $
f^{-1}(r_{2})$ are of the same homotopy type for any
$r_{1},r_{2}\in O$.}
\end{thm}

%%====================================================================================

\begin{lem}\label{a115}
\emph{\cite{Spanier} Let   $R_{i},M_{j}:S^I\longrightarrow S^I$ be
maps, where $i,j=1,2,3,4$, which are defined by
$R_{1}(\alpha)=\widetilde{\alpha(0)}$,
$M_{1}(\alpha)=\alpha\star\overline{\alpha}$,
     $R_{2}(\alpha)=\widetilde{\alpha(1)}$,  $
M_{2}(\alpha)=\overline{\alpha}\star\alpha $,
$M_{3}(\alpha)=\alpha$, $R_{3}(\alpha)=\widetilde{\alpha(0)}\star
\alpha$, $M_{4}(\alpha)=\alpha$ and  $
R_{4}(\alpha)=\alpha\star\widetilde{\alpha(1)}$ for all $\alpha\in
S^I$. Then $R_{1}$, $M_{1}$ are homotopic by homotopy $H$ which
has the following property:
\begin{eqnarray}\label{eeq1}
[H(\alpha,r)](1) =\alpha(0) \quad\mbox{ for }\mbox{ }r\in I,
\alpha\in S^I,
  \end{eqnarray}
  and $R_{i}$, $M_{i}$, ($i=2,3,4$),
  are homotopic by homotopy $G_i$ which
has the following property:
\begin{eqnarray}\label{eeq2}
[G_i(\alpha,r)](1)=\alpha(1) \quad\mbox{ for }\mbox{ }r\in I,
\alpha\in S^I.
  \end{eqnarray}   }
\end{lem}

%%====================================================================================
\begin{lem}\label{a116}
\emph{  \cite{Gray} Consider   Figure 1 which involves abelian
groups and homomorphisms}
\begin{figure}[h!]
   \begin{center}
 \begin{displaymath} \xymatrix{
 G_{1} \ar[rr]^{h_{1}} \ar[d]^{\psi_{1}}&&
           G_{2}\ar[rr]^{h_{2}}\ar[d]^{\psi_{2}}&& G_{3}\ar[rr]^{h_{3}}\ar[d]^{\psi_{3}}&&G_{4}\ar[rr]^{h_{4}}\ar[d]^{\psi_{4}}&& G_{5}\ar[d]^{\psi_{5}}                      \\
G'_{1}\ar[rr]_{h'_{1}}   &&G'_{2}\ar[rr]_{h'_{2}}
&&G'_{3}\ar[rr]_{h'_{3}}
    &&G'_{4}\ar[rr]_{h'_{4}}&&
           G'_{5} }
\end{displaymath}
  \end{center}
  \vspace{-2mm}
 \centerline{Figure 1}
 \label{apic4}
\end{figure}
 \emph{ such that
$\psi_{i+1}\circ h_{i}=h'_{i}\circ \psi_{i}$ for all $i=1,2,3,4$.
If $\psi_{1},\psi_{2},\psi_{4}$ and  $\psi_{5}$ are isomorphisms,
then $\psi_{3}$ is an isomorphism.}
\end{lem}
%%%%###########################################################################

\section{ Fibrations $\Gamma(f,s_{o})$ and $\Sigma(f)$}

\noindent In this section, we shall introduce the notions of
  fibrations $\Sigma(f )$ and
   $\Gamma(f ,s_{o})$
   which are induced by   fibration $f:S\longrightarrow O$ with
   some results about their properties.

%%%%------------------------------------------------------------------------------
\noindent The \emph{functors} $\Gamma$ and $\Sigma$ are defined as
follows:\newline
\[\Gamma(S,F,s_{o})=\{\alpha\in S^I: \alpha(0)=s_{o}, \mbox{ }
\alpha(1)\in F\}  \]and
\[\Sigma(S,F)=\{\alpha\in S^I: \alpha(0)\in F,\mbox{ }
\alpha(1)\in F\} \] for any   subspace $F$ of  any topological
space $S$, where  $s_{o}\in F$.\\

\noindent Let $f:S\longrightarrow O$ be a fibration with a fiber
space $F_{r_{o}}$.
 We will define   two fibrations $\Gamma(f ,s_{o})$  and  $\Sigma(f)$   on the
functors $\Gamma$ and $\Sigma$, respectively,  induced by $f$  as
follow:\newline
   $\Gamma(f ,s_{o})$ will denote the fibration
   $\Psi_{s_{o}}:\Gamma(S,F_{r_{o}},s_{o})\longrightarrow
    F_{r_{o}} $  given by
    \[\Psi_{s_{o}}(\alpha)=\alpha(1)\quad\mbox{ for}\mbox{ } \alpha\in
    \Gamma(S,F_{r_{o}},s_{o}) \] and we say $\Gamma(f ,s_{o})$ is a fibration induced by
      $f$, which has fiber space $\Psi_{s_{o}}^{-1}(s_{o})= \Omega(S,s_{o})$ over a point $s_{o}\in
    F_{r_{o}}$. \newline
$\Sigma(f)$ will be denote the fibration
$\Phi:\Sigma(S,F_{r_{o}})\longrightarrow  F_{r_{o}}\times
F_{r_{o}} $ given by
\[\Phi(\alpha)=[\alpha(0),\alpha(1)]\quad\mbox{ for}
   \mbox{ }\alpha\in
    \Gamma(S,F_{r_{o}})\] and we say $\Sigma(f)$ is a fibration induced by $f$, which has fiber space
    $\Phi^{-1}[(s_{o},s_{o})]=  \Omega(S,s_{o})$     over  a point $(s_{o},s_{o})\in  F_{r_{o}}\times F_{r_{o}}
    $.

%%====================================================================================

\begin{lem}\label{a117}
Let $f:S\longrightarrow O$ be a fibration. Then the maps
$D,D_{o}:\bigtriangleup f \longrightarrow S$ defined by
\[D(s,\alpha)=L_{f}[L_{f}(s,\alpha)(1),\overline{\alpha}](1)\quad\mbox{and}\quad D_{o}(s,\alpha)=s,\]
for all $(s,\alpha)\in \bigtriangleup
f$,
 are  homotopic.
 \end{lem}
\begin{proof}
 For $\alpha \in O^I$ and $r\in I$,   define  paths
$\alpha_{r},\alpha'_{r}$ and $ \alpha''_{r}$ in $O$ by
\[\alpha_{r}(t)=\alpha(rt),\quad\alpha'_{r}(t)=\alpha[r+(1-r)t] \quad\mbox{and}\quad \alpha''_{r}(t)=\alpha[2r(1-t)],\]
 for all $t\in I$.
Define two homotopies $H: \bigtriangleup f \times I\longrightarrow
S$ by
\[H[(s,\alpha),t]=L_{f}[L_{f}(s,\alpha_{t})(1),\alpha'_{t}](1) \quad\mbox{ for}
   \mbox{ }t\in I, (s,\alpha)\in \bigtriangleup f,\] and a homotopy
   $G:O^I \times I\longrightarrow O^I$ by
\[ [G(\alpha,r)](t) = \left\{
  \begin{array}{c l}
    \alpha_{r}(t) & \mbox{for \quad$ 0 \leq t \leq 1/2$},\\
    \alpha''_{r}(t)  & \mbox{for\quad $1/2 \leq t \leq 1$},
  \end{array}
\right. \] for all  $\alpha\in O^I, r\in I$.
   Hence define a homotopy $F: \bigtriangleup f \times I\longrightarrow
S$ by
\[F[(s,\alpha),t]=H[(s,G(\alpha,t)),1/2] \quad\mbox{ for}
   \mbox{ } t\in I, (s,\alpha)\in \bigtriangleup f.\]
   By the  regularity for $L_{f}$     we observe that
for $(s,\alpha)\in \bigtriangleup f$,
\begin{eqnarray*}
F[(s,\alpha),1]=H[(s,G(\alpha,1)),1/2]&=&H[(s,\alpha\star\overline{\alpha}),1/2]\\
   &=& L_{f}\{K[s,(\alpha\star\overline{\alpha})_{1/2}],(\alpha\star\overline{\alpha})'_{1/2}\}(1) \\
   &=& L_{f}[K(s,\alpha),\overline{\alpha}](1) \\
   &=& L_{f}[L_{f}(s,\alpha)(1),\overline{\alpha}](1)  \\
   &=& D(s,\alpha)
\end{eqnarray*}
  for all $(s,\alpha)\in \bigtriangleup f$, and
  \begin{eqnarray*}
F[(s,\alpha),0]=H[(s,G(\alpha)),0](1/2)&=&H[(s,\widetilde{\alpha(0)}),1/2]\\
   &=& L_{f}\{K[s,\widetilde{\alpha(0)}_{1/2}],\widetilde{\alpha(0)}'_{1/2}\}(1) \\
   &=& L_{f}\{K[s,\widetilde{\alpha(0)}],\widetilde{\alpha(0)}\}(1) \\
   &=& L_{f}[L_{f}(s,f\circ\widetilde{s})(1),f\circ\widetilde{s}](1)  \\
   &=& L_{f}(s,f\circ\widetilde{s})(1)\\
   &=& s= D_{o}(s,\alpha)
\end{eqnarray*} for all $(s,\alpha)\in \bigtriangleup f$.
Hence $D$ and $D_{o}$ are  homotopic.
 \end{proof}
\noindent In the proof of Lemma  above we get that the  homotopy
$F$ has the following property:\newline
\begin{eqnarray}\label{eeq3}
f\{F[(s,\alpha),t]\}   &=&\alpha(0) \quad\mbox{ for}
   \mbox{ } (s,\alpha)\in \bigtriangleup f.
   \end{eqnarray}
%%%=======================================================================================
\begin{pro}\label{a118}
For any fibration $f:S\longrightarrow O$ with fiber space $
F_{r_{o}}$, the following statements are true:\newline 1.
$\Sigma(S,F_{r_{o}})\simeq  \Omega(O,r_{o})\times F_{r_{o}}
$;\newline 2.  $\Gamma(S,F_{r_{o}},s_{o})\simeq  \Omega(O,r_{o})
$\quad for all $s_{o} \in F_{r_{o}} $.
\end{pro}
\begin{proof}
1. Define      a map $N:\Sigma(S,F_{r_{o}})\longrightarrow
 \Omega(O,r_{o})\times F_{r_{o}}$ by
\[N(\alpha)=[f\circ \alpha,\alpha(0)] \quad\mbox{for}\mbox{ }
\alpha\in \Sigma(S,F_{r_{o}}), \] and      a map $M:
\Omega(O,r_{o})\times F_{r_{o}}\longrightarrow
\Sigma(S,F_{r_{o}})$ by
\[M(\alpha,s)=L_{f}(s,\alpha) \quad\mbox{for}\mbox{ } (\alpha,s)\in
 \Omega(O,r_{o})\times F_{r_{o}}.\] Then we have that
\begin{eqnarray*}
  (N\circ M)(\alpha,s) &=& N[L_{f}(s,\alpha)] \\
   &=& \{f[L_{f}(s,\alpha)],L_{f}(s,\alpha)(0)\} \\
   &=& (\alpha,s)=id_{ \Omega(O,r_{o})\times
F_{r_{o}}}(\alpha,s)
\end{eqnarray*}
for all $(\alpha,s)\in  \Omega(O,r_{o})\times F_{r_{o}}$. That is,
$N\circ M=id_{ \Omega(O,r_{o})\times F_{r_{o}}}$. By Lemma
\ref{a117}, we have that the composition  map $M\circ
N:\Sigma(S,F_{r_{o}})\longrightarrow \Sigma(S,F_{r_{o}})$ given by
\[(M\circ N)(\alpha)=L_{f}[\alpha(0),f\circ \alpha]  \quad\mbox{for}\mbox{ } \alpha\in
\Sigma(S,F_{r_{o}}) \] is   homotopic to the identity  map
$id_{\Sigma(S,F_{r_{o}})}$. Therefore
\[\Sigma(S,F_{r_{o}})\simeq  \Omega(O,r_{o})\times F_{r_{o}}.\]
 2.
 Let $s_{o}\in
 F_{r_{o}} $.  Define      a map $R:\Gamma(S,F_{r_{o}},
s_{o})\longrightarrow  \Omega(O,r_{o})$ by
\[R(\alpha)=f\circ \alpha \quad\mbox{for}\mbox{ }
\alpha\in \Gamma(S,F_{r_{o}},s_{o}), \] and     a map $D:
\Omega(O,r_{o})\longrightarrow \Gamma(S,F_{r_{o}},s_{o})$ by
\[D(\alpha)=L_{f}(s_{o},\alpha) \quad\mbox{for}\mbox{ } \alpha\in
 \Omega(O,r_{o}).\] Then we have
\begin{eqnarray*}
  (R\circ D)(\alpha) &=& R(L_{f}(s_{o},\alpha)) \\
   &=& f[L_{f}(s_{o},\alpha)] \\
   &=& \alpha=id_{ \Omega(O,r_{o})}(\alpha)
\end{eqnarray*}  for all $\alpha\in
    \Omega(O,r_{o})$. That is, $R\circ D=id_{ \Omega(O,r_{o})}$. By Lemma \ref{a117}, we get that the composition
 map $D\circ R:\Gamma(S,F_{r_{o}},s_{o})\longrightarrow
\Gamma(S,F_{r_{o}},s_{o})$ given by
\[(D\circ R)(\alpha)=L_{f}(s_{o},f\circ \alpha)=L_{f}(\alpha(0),f\circ \alpha)  \] for all $\alpha\in
\Gamma(S,F_{r_{o}},s_{o})$ is  homotopic to the identity  map
$id_{\Gamma(S,F_{r_{o}},s_{o})}$. Therefore
\[\Gamma(S,F_{r_{o}},s_{o})\simeq  \Omega(O,r_{o}),\] for all $ s_{o}\in
 F_{r_{o}}$.
 \end{proof}
%%%=======================================================================================
\noindent There are several fibrations $\Gamma(f,s_{o})$ according
to the number of points in $ F_{r_{o}}$. But when we let $
F_{r_{o}}$ pathwise  connected, then the set of  fiber homotopy
equivalence classes of the collection set of all these
  fibrations   will be a single. As it is clear in the
following theorem.
%%====================================================================================

\begin{thm}\label{a1110}
 Let $f:S\longrightarrow O$ be a fibration
with a pathwise  connected fiber space $ F_{r_{o}}$. Then the
  fibration $\Gamma(f,s_{o})$ is determined up
to      a fiber homotopy equivalence class. That is,
$\Gamma(f,s_{o})$ and $\Gamma(f,s'_{o})$ are fiber homotopy
equivalent  for all $s_{o}, s'_{o}\in  F_{r_{o}}$.
\end{thm}
 \begin{proof}
 Let $s_{o}, s'_{o}\in  F_{r_{o}}$. Since  $ F_{r_{o}}$ is  a pathwise
 connected then  there is  path $\beta:I\longrightarrow
F_{r_{o}}$ between  $s_{o}$ and $s'_{o}$.
 Now let us   to define two  fiber maps, we can
define    the map $h:\Gamma(S,F_{r_{o}},s_{o})\longrightarrow
\Gamma(S,F_{r_{o}},s'_{o})$ by
\[h(\alpha)=\overline{\beta}\star\alpha \quad\mbox{for}\mbox{ }\alpha\in
    \Gamma(S,F_{r_{o}},s_{o}), \] and  a map
    $g:\Gamma(S,F_{r_{o}},s'_{o})\longrightarrow
\Gamma(S,F_{r_{o}},s_{o})$ by
\[g(\alpha)=\beta\star\alpha \quad\mbox{for}\mbox{ }\alpha\in
    \Gamma(S,F_{r_{o}},s'_{o}).\]Then we have
\[\Psi_{s'_{o}}[h(\alpha)]=(\overline{\beta}\star\alpha)(1)=\alpha(1)=\Psi_{s_{o}}(\alpha)  \]  and
    \[\Psi_{s_{o}}[g(\alpha)]=(\beta\star\alpha)(1)=\alpha(1)=\Psi_{s'_{o}}(\alpha) \] for all $\alpha\in
    \Gamma(S,F_{r_{o}},s_{o})$. That is, $h$ and $g$ are
     fiber maps.

\noindent Now from Lemma \ref{a115} and Equations \ref{eeq1},
\ref{eeq2}, we observe that the composition  map $g\circ
    h:\Gamma(S,F_{r_{o}},s_{o})\longrightarrow\Gamma(S,F_{r_{o}},s_{o})$
    given by
\[(g\circ h)(\alpha)=(\beta\star\overline{\beta})\star \alpha \quad\mbox{for}\mbox{ }\alpha\in
    \Gamma(S,F_{r_{o}},s_{o}) \] is    fiber homotopic to the
    identity  map $id_{\Gamma(S,F_{r_{o}},s_{o})}$ and  the composition  map $h\circ
    g:\Gamma(S,F_{r_{o}},s'_{o})\longrightarrow\Gamma(S,F_{r_{o}},s'_{o})$
    given by
\[(h\circ g)(\alpha)=(\overline{\beta}\star \beta)\star \alpha \quad\mbox{for}\mbox{ }\alpha\in
    \Gamma(S,F_{r_{o}},s'_{o})\] is    fiber homotopic to the
    identity  map $id_{\Gamma(S,F_{r_{o}},s'_{o})}$.  Therefore
$\Gamma(f,s_{o})$ and $\Gamma(f,s'_{o})$ are fiber homotopy
equivalent.
 \end{proof}

%%====================================================================================
\noindent Here we give some concepts which will be used in the
next sections.
%%====================================================================================
\begin{defn}\label{a1111}
\emph{ Let $f:S \longrightarrow O$ be a fibration with fiber space
$F_{r_{o}}=f^{-1}(r_{o})$, where $r_{o}\in O$.  By the
\emph{$Lf-$function for  fibration $f$ induced by a lifting
function $L_{f}$ } we mean a map $\Theta_{L_{f}}:
\Omega(O,r_{o})\times F_{r_{o}} \longrightarrow  F_{r_{o}} $ which
is defined by
\[\Theta_{L_{f}}(\alpha,s)=L_{f}(s,\alpha)(1)\quad \mbox{for } \mbox{ }s\in F_{r_{o}}, \alpha\in  \Omega(O,r_{o}).\]
}
\end{defn}

%%====================================================================================
\noindent Henceforth, we will denote by $[S,f,O,
F_{r_{o}},\Theta_{L_{f}}]$
 the  regular fibration  $f:S \longrightarrow O$ with   an  $Lf-$function
 $\Theta_{L_{f}}:
   \Omega(O,r_{o})\times F_{r_{o}} \longrightarrow  F_{r_{o}}$, induced by the
 lifting function $L_{f} $ and with a  fiber space
$F_{r_{o}}=f^{-1}(r_{o})$, where  $r_{o}\in O$.
%%====================================================================================

\begin{defn}\label{a1112}
\emph{Let $[S,f,O, F_{r_{o}},\Theta_{L_{f}}]$ be a fibration. For
$s_{o}\in F_{r_{o}}$, the map $R:  \Omega(O,r_{o}) \longrightarrow
 F_{r_{o}} $ defined by
\[R(\alpha)=\Theta_{L_{f}}(\alpha,s_{o}) \quad\mbox{for }
   \mbox{ }\alpha\in
 \Omega(O,r_{o})\] is called an \emph{$Lf-$restriction for the fibration
$f$}  and we denote it by $f^{s_{o}}$.}
\end{defn}

%%====================================================================================
\begin{exa}\label{a1113}
\emph{ The first  fibration $P_{1}: O\times S \longrightarrow O$
has a regular lifting function $L_{P_{1}}:
 \bigtriangleup P_{1} \longrightarrow  (O\times S)^I$ defined
by
\[L_{P_{1}}[(b,s),\alpha](t)= (\alpha(t),s) \quad\mbox{for   }
   \mbox{ }t\in I, [(b,s),\alpha]\in \bigtriangleup L_{P_{1}}.\]
Then the $Lf-$function $\Theta_{L_{P_{1}}}$ for      fibration
$P_{1}$ induced by $L_{P_{1}}$ will be given by
\[\Theta_{L_{P_{1}}}(\alpha,x)=x \quad\mbox{for   }
   \mbox{ } x\in
F_{r_{o}}, \alpha\in  \Omega(O,r_{o}).\] The $Lf-$restriction
$P^{s_{o}}_{1}$ for the fibration $P_{1}$ will be given  by
\[P^{s_{o}}_{1}(\alpha)=s_{o}\quad \mbox{for all }\mbox{ }
\alpha\in  \Omega(O,r_{o}).\]}
\end{exa}
%%====================================================================================

\begin{defn}\label{a1114}
\emph{Let  $[S_1,f_1,O,F^1_{r_{o}},\Theta_{L_{f_{1}}}]$ and
$[S_2,f_2,O,F^2_{r_{o}},\Theta_{L_{f_{2}}}]$ be two  fibrations.
The $Lf-$functions $\Theta_{L_{f_{1}}}$ and $\Theta_{L_{f_{2}}}$
are said to be \emph{conjugate }if there is $g\in
 H(F^{1}_{r_{o}},F^{2}_{r_{o}})$ such that
\[\Theta_{L_{f_{1}}}\simeq_{S}\widetilde{\overline{g}}\circ
\Theta_{L_{f_{2}}}\circ(id_{ \Omega(O,r_{o})}\times g).\] We say
that $f_{1}$ and $f_{2}$ \emph{have conjugate $Lf-$restrictions}
if there is $g\in  H(F^{1}_{r_{o}},F^{2}_{r_{o}})$  such that \[
f^{s_{o}}_{1}\simeq_{S} \widetilde{\overline{g}}\circ
f^{g(s_{o})}_{2}\] where $s_{o}\in F_{r_{o}}$,
$H(F^{1}_{r_{o}},F^{2}_{r_{o}})$ is the set of all homotopy
equivalences from $F^{1}_{r_{o}}$ into $F^{2}_{r_{o}}$ and
$\widetilde{\overline{g}}$ is the inverse homotopy of $g$. }
\end{defn}
%%====================================================================================

\noindent If two fibrations  have conjugate $Lf-$functions,   they
also have conjugate $Lf-$restrictions.

%%%%#####################################################################################
\section{ Fibration  $ \Gamma(f,s_{o})$ and $Lf-$restriction }

\noindent In this section, we are going to introduce the role of
homotopy sequences   of fibrations (using $Lf-$restriction ) in
satisfying FHE between two fibrations $\Gamma(f_{1},s_{o})$ and
$\Gamma(f_{2},s'_{o})$ which are induced by two fibrations
$[S_1,f_1,O,F^1_{r_{o}},\Theta_{L_{f_{1}}}]$ and
$[S_2,f_2,O,F^2_{r_{o}},\Theta_{L_{f_{2}}}]$ over a common
base $O$.\\

%%====================================================================================
\noindent In the following theorem, we show that for two
  fibrations $f_{1}$  and $f_{2}$   with  conjugate
$Lf-$restrictions $f^{s_{o}}_{1}$ and $f^{g(s_{o})}_{2}$ by $g\in
 H(F^{1}_{r_{o}},F^{2}_{r_{o}})$, there are two  fiber maps
 between two   fibrations  $\Gamma(f_{1},s_{o})$  and
$\Gamma(f_{2},g(s_{o}))$.

%%====================================================================================

\begin{thm}\label{a1115}
Let $[S_1,f_1,O,F^1_{r_{o}},\Theta_{L_{f_{1}}}]$ and
$[S_2,f_2,O,F^2_{r_{o}},\Theta_{L_{f_{2}}}]$ be two   fibrations
with
  conjugate $Lf-$restrictions $f^{s_{o}}_{1}$ and
$f^{g(s_{o})}_{2}$ by $g\in H(F^{1}_{r_{o}},F^{2}_{r_{o}})$, where
$s_{o} \in F^{1}_{r_{o}} $. Then there are two fiber maps
\[h:\Gamma(S_{1},F^{1}_{r_{o}},s_{o})\longrightarrow
\Gamma(S_{2},F^{2}_{r_{o}},g(s_{o})) \] and
\[k:\Gamma(S_{2},F^{2}_{r_{o}},g(s_{o}))\longrightarrow
\Gamma(S_{1},F^{1}_{r_{o}},(\widetilde{\overline{g}}\circ
g)(s_{o}))\] over $g$ and $\widetilde{\overline{g}}$,
respectively. That is, Figure 2 is   commutative.
\begin{figure}[h!]
   \begin{center}
 \begin{displaymath} \xymatrix{
 \Gamma(S_{1},F^{1}_{r_{o}},s_{o}) \ar[rr]^h \ar[d]^{\Psi_{s_{o}}}&&
           \Gamma(S_{2},F^{2}_{r_{o}},g(s_{o}))\ar[rr]^k\ar[d]^{\Psi_{g(s_{o})}}&& \Gamma(S_{1},F^{1}_{r_{o}},(\widetilde{\overline{g}}\circ g)(s_{o}))\ar[d]^{\Psi_{(\widetilde{\overline{g}}\circ g)(s_{o})}} &    \\
  F^{1}_{r_{o}}  \ar[rr]_{g}   &&  F^{2}_{r_{o}} \ar[rr]_{\widetilde{\overline{g}}} &&
  F^{1}_{r_{o}} }
\end{displaymath}
  \end{center}
  \vspace{-2mm}
 \centerline{Figure 2}
 \label{apic5}
\end{figure}
 Further the groups
$\pi_{n}(S_{1},s_{o})$ and $\pi_{n}(S_{2},g(s_{o}))$ are
isomorphic for  a positive integer $n>1$.
\end{thm}
 \begin{proof}
  By hypothesis,
$f^{s_{o}}_{1}\simeq \widetilde{\overline{g}}\circ
f^{g(s_{o})}_{2}$. This implies   $g\circ f^{s_{o}}_{1}\simeq
f^{g(s_{o})}_{2}$.  Then there is
 a homotopy $G: \Omega(O,r_{o})\times I\longrightarrow
F^{2}_{r_{o}}$  such that \[
G(\alpha,0)=f^{g(s_{o})}_{2}(\alpha)=\Theta_{L_{f_{2}}}(\alpha,g(s_{o}))
\] and  \[ G(\alpha,1)=(g\circ
f^{s_{o}}_{1})(\alpha)=g[\Theta_{L_{f_{1}}}(\alpha,s_{o})] \]for
all $\alpha\in  \Omega(O,r_{o})$. For $\alpha \in S_{1}^I$ and
$r\in I$, we can define a path $\alpha_{r} \in S_1^I$ by
\[ \alpha_{r}(t)= \left\{
  \begin{array}{c l}
    \alpha(t) & \mbox{for} \quad 0\leq t\leq r,\\
    \alpha(r)   & \mbox{for}\quad r\leq t \leq 1.
  \end{array}
\right. \] For $\beta = f_{1}\circ \alpha$ and $r\in I$, we can
define the path $\beta^{1-r}\in O^I$ by
\[ \beta^{1-r}(t)= \left\{
  \begin{array}{c l}
    \beta(r+t) & \mbox{for} \quad 0\leq t\leq 1-r,\\
    \beta(1)   & \mbox{for}\quad 1-r\leq t \leq 1.
  \end{array}
\right. \] Define   a homotopy $H:S_{1}^I\times I\longrightarrow
 S_{1}^I$ by
\[ [H(\alpha,r)](t)= \left\{
  \begin{array}{c l}
    \alpha_{r}(t),& \mbox{for} \quad 0\leq t\leq r,\\
    L_{f_{1}}(\alpha(r),\beta^{1-r})(t-r) , & \mbox{for}\quad r\leq t \leq 1,
  \end{array}
\right. \] for all $r\in I, \alpha\in S_{1}^I$. We get that
\[H(\alpha,0)=L_{f_{1}}(\alpha(0),f_{1}\circ
\alpha)\quad\mbox{and}\quad H(\alpha,1)=\alpha \quad\mbox{for
}\mbox{  }\alpha\in S_{1}^I.\]

\noindent For $\alpha\in \Gamma(S_{1},F^{1}_{r_{o}},s_{o})$, let
$M(\alpha)$ be a path in $ \Gamma(S_{1},F^{1}_{r_{o}},s_{o})$
defined by
\[M(\alpha)(r)=[H(\alpha,r)](1) \quad\mbox{for}\mbox{ } r\in
I.\] Hence we can define a map
$M':\Gamma(S_{1},F^{1}_{r_{o}},s_{o})\longrightarrow
(F^{2}_{r_{o}})^I$ by
\[M'(\alpha)=g[M(\alpha)] \quad\mbox{for}\mbox{ }\alpha\in
\Gamma(S_{1},F^{1}_{r_{o}},s_{o}), \] and  a map
$L'_{f_{2}}:\Gamma(S_{1},F^{1}_{r_{o}},s_{o})\longrightarrow\Gamma(S_{2},F^{2}_{r_{o}},g(s_{o}))$
by
\[L'_{f_{2}}(\alpha)=L_{f_{2}}(g(s_{o}),f_{1}\circ\alpha) \quad\mbox{for}\mbox{ }\alpha\in
\Gamma(S_{1},F^{1}_{r_{o}},s_{o}).\] Now define the
   map
$h:\Gamma(S_{1},F^{1}_{r_{o}},s_{o})\longrightarrow
\Gamma(S_{2},F^{2}_{r_{o}},g(s_{o}))$  by
\begin{eqnarray*}\label{q741}
  h(\alpha)&=&[L'_{f_{2}}(\alpha)\star G(f_{1}\circ \alpha)]\star
M'(\alpha) \quad\mbox{for}\mbox{ }\alpha \in
\Gamma(S_{1},F^{1}_{r_{o}},s_{o}).
\end{eqnarray*}
 Then $h$ is a well-defined as  a continuous map. Since
 \begin{eqnarray*}
   (\Psi_{g(s_{o})}\circ h)(\alpha)
    &=&\{[L'_{f_{2}}(\alpha)\star G(f_{1}\circ \alpha)]\star
M'(\alpha)\}(1) \\
    &=& M'(\alpha)(1)=g[M(\alpha)(1)]=g[\alpha(1)]\\
    &=& (g\circ \Psi_{s_{o}})(\alpha)
 \end{eqnarray*} for all $\alpha \in
\Gamma(S_{1},F^{1}_{r_{o}},s_{o})$. That is,
  $\Psi_{g(s_{o})}\circ h=g\circ \Psi_{s_{o}}$. Hence $h$ is fiber map over
  $g$.\\

\noindent  Now we   prove that $\pi_{n}(S_{1},s_{o})$ and
$\pi_{n}(S_{2},g(s_{o}))$ are isomorphic for  a positive integer
$n>1$.   In Figure 3, define the map
$N_{1}:\Gamma(S_{1},F^{1}_{r_{o}},s_{o})\longrightarrow
 \Omega(O,r_{o})$   by
\[N_{1}(\alpha)=f_{1}\circ \alpha \quad\mbox{for} \mbox{ } \alpha\in
\Gamma(S_{1},F^{1}_{r_{o}},s_{o}), \] and  a map
$N_{2}:\Gamma(S_{2},F^{2}_{r_{o}},g(s_{o}))\longrightarrow
 \Omega(O,r_{o})$ by \[N_{2}(\alpha)=f_{2}\circ \alpha \quad\mbox{for }
\mbox{ } \alpha\in \Gamma(S_{2},F^{2}_{r_{o}},g(s_{o})).\]
\begin{figure}[h!]
   \begin{center}
 \begin{displaymath} \xymatrix{
\Gamma(S_{1},F^{1}_{r_{o}},s_{o}) \ar[d]_{h}
           \ar[drrrrr]^{N_{1}}&\\
 \Gamma(S_{2},F^{2}_{r_{o}},g(s_{o}))\ar[rrrrr]_{N_{2}}   &&&&&  \Omega(O,r_{o}) }
\end{displaymath}
  \end{center}
  \vspace{-2mm}
 \centerline{Figure 3}
 \label{apic6}
\end{figure}
From the definition of    fiber map $h$, we observe that
$[(N_{2}\circ h)(\alpha)](t)=r_{o}$ at $t=1/4$ and  for $t\in
[0,1/4]$,
\begin{eqnarray*}
  [(N_{2}\circ h)(\alpha)](t) =[N_{2}(h(\alpha))](t) &=& f_{2}[L'_{f_{2}}(\alpha)](4t) \\
   &=&f_{2}[L_{f_{2}}(g(s_{o}),f_{1}\circ \alpha)](4t)\\
  &=&(f_{1}\circ \alpha)(4t)= [N_{1}(\alpha)](4t).
\end{eqnarray*} That is,  $(N_{2}\circ h)(\alpha)\neq N_{1}(\alpha)$ concentrated on the interval $[0,1]$ but
$(N_{2}\circ h)(\alpha)=N_{1}(\alpha)$ concentrated on the
interval $[0,1/4]$. This implies that $N_{2}\circ h\simeq N_{1}$.
Hence Figure 3 is not commutative in the usual sense but it is a
 homotopy commutative. That is,   Figure 4 is a commutative
 \begin{figure}[h!]
   \begin{center}
 \begin{displaymath} \xymatrix{
\pi_{n}(\Gamma(S_{1},F^{1}_{r_{o}},s_{o}),\widetilde{s_{o}})
\ar[d]_{\widehat{h}}
           \ar[drrrrr]^{\widehat{N_{1}}}&\\
 \pi_{n}(\Gamma(S_{2},F^{2}_{r_{o}},g(s_{o})),\widetilde{g(s_{o})})\ar[rrrrr]_{\widehat{N_{2}}}   &&&&&\pi_{n}( \Omega(O,r_{o}),\widetilde{r_{o}})}
\end{displaymath}
  \end{center}
  \vspace{-2mm}
 \centerline{Figure 4}
 \label{apic7}
\end{figure}
i.e.,
\begin{eqnarray}\label{eeq4}
 \widehat{N_{2}}\circ \widehat{h}&=&\widehat{ N_{1}}.
\end{eqnarray}

\noindent By   the part 2 in Proposition  \ref{a118},  the  maps
$N_{1}$ and $N_{2}$  are
 homotopy equivalences of
$\Gamma(S_{1},F^{1}_{r_{o}},s_{o})$ into $   \Omega(O,r_{o})$ and
of $\Gamma(S_{2},F^{2}_{r_{o}},g(s_{o}))$ into $
\Omega(O,r_{o})$, respectively. Hence
\[\widehat{N_{1}}:\pi_{n}(\Gamma(S_{1},F^{1}_{r_{o}},s_{o}),\widetilde{s_{o}})\longrightarrow
\pi_{n}( \Omega(O,r_{o}),\widetilde{r_{o}}) \] and \[\widehat{
N_{2}}:\pi_{n}(\Gamma(S_{2},F^{2}_{r_{o}},g(s_{o})),\widetilde{g(s_{o})})\longrightarrow
\pi_{n}( \Omega(O,r_{o}),\widetilde{r_{o}})\] are isomorphisms for
 a positive integer $n>0$. By Equation \ref{eeq4}, we get that
\[\widehat{h}:\pi_{n}(\Gamma(S_{1},F^{1}_{r_{o}},s_{o}))\longrightarrow
\pi_{n}(\Gamma(S_{2},F^{2}_{r_{o}},g(s_{o})))\] is an isomorphism
for  a positive integer $n>0$. Consider
\[h_{o}: \Omega(S_{1},s_{o})\longrightarrow  \Omega(S_{2},g(s_{o}))\] is a restriction of $h$ on
$\Psi_{s_{o}}^{-1}(s_{o})=  \Omega(S_{1},s_{o})$. Now we can
integrate the homotopy sequences of   fibrations $\Psi_{s_{o}}$
and $\Psi_{g(s_{o})}$ in Figure 5,
\begin{figure}[h!]
   \begin{center}
 \begin{displaymath} \xymatrix{
\ar@{..}[r]& \pi_{n+1}(\Gamma(S_{1}))
\ar[rr]^{\widehat{\Psi_{s_{o}}}} \ar[d]^{\widehat{h}}&&
           \pi_{n+1}(F^{1}_{r_{o}},s_{o})\ar[rr]^{(\partial_{1})_{\bullet}\quad}\ar[d]^{\widehat{g}}&& \pi_{n}( \Omega(S_{1},s_{o}),\widetilde{s_{o}})\ar[d]^{\widehat{h_{o}}}                       \\
\ar@{..}[r]&
\pi_{n+1}(\Gamma(S_{2}))\ar[rr]_{\widehat{\Psi_{g(s_{o})}}}
&&\pi_{n+1}(F^{2}_{r_{o}},g(s_{o}))\ar[rr]_{(\partial_{1})_{\bullet}\quad}
&&\pi_{n}( \Omega(S_{2},g(s_{o})),\widetilde{g(s_{o})})&}
\end{displaymath}
\begin{displaymath} \xymatrix{
\ar[rr]^{\widehat{j_{1}}\quad}&&\pi_{n}(\Gamma(S_{1}))\ar[rr]^{\widehat{\Psi_{s_{o}}}}\ar[d]^{\widehat{h}}&&
\pi_{n}(F^{1}_{r_{o}},s_{o})\ar@{..}[r] \ar[d]^{\widehat{g}} & \\
 \ar[rr]_{\widehat{j_{2}}\quad}
    &&\pi_{n}(\Gamma(S_{2}))\ar[rr]_{\widehat{\Psi_{g(s_{o})}}}&&
           \pi_{n}(F^{2}_{r_{o}},g(s_{o}))\ar@{..}[r]& & }
\end{displaymath}
  \end{center}
  \vspace{-2mm}
 \centerline{Figure 5}
 \label{apic8}
\end{figure}
\newline
 where $j_{1}, j_{2}$ are inclusion
 maps, $\partial_{1}$ and $\partial_{2}$ are boundary
operators and
\begin{eqnarray*}
  (\partial_{1})_{\bullet}&=& \partial_{1}\circ (\widehat{\Psi_{s_{o}}})^{-1}, \\
  (\partial_{2})_{\bullet} &=& \partial_{2}\circ
(\widehat{\Psi_{g(s_{o})}})^{-1},\\
  \pi_{n}(\Gamma(S_{1})) &=& \pi_{n}(\Gamma(S_{1},F^{1}_{r_{o}},s_{o}),\widetilde{s_{o}}), \\
  \pi_{n}(\Gamma(S_{2})) &=&
  \pi_{n}(\Gamma(S_{2},F^{2}_{r_{o}},g(s_{o})),\widetilde{g(s_{o}})),
\end{eqnarray*}
 for a positive integer $n>0$.

 \noindent In Figure 5,  we observe  that $\widehat{j_{1}}$,
 $\widehat{j_{2}}$, $(\partial_{1})_{\bullet}$,  $(\partial_{2})_{\bullet}$,
 $\widehat{\Psi_{s_{o}}}$ and $\widehat{\Psi_{g(s_{o})}}$ are homomorphisms. Since
 $\widehat{h}$ and $\widehat{g}$ are isomorphisms, then
 Lemma \ref{a116} shows that for a positive integer $n>0$,
\[\widehat{h_{o}}:\pi_{n}( \Omega(S_{1},s_{o}),\widetilde{s_{o}})\longrightarrow \pi_{n}( \Omega(S_{2},g(s_{o})),\widetilde{g(s_{o})}) \]
is an  isomorphism.
 Since  $\pi_{n+1}(S_{1},s_{o})$ is isomorphic to
 $\pi_{n}( \Omega(S_{1},s_{o}),\widetilde{s_{o}})$ and
    $\pi_{n+1}( S_{2}, g(s_{o})) $ is isomorphic to $\pi_{n}( \Omega(S_{2},g(s_{o})),\widetilde{g(s_{o})})$ for a positive integer $n>0$, then $\pi_{n}(S_{1},s_{o})$ is isomorphic to
     $\pi_{n}( S_{2}, g(s_{o})) $ for a positive integer $n>1$.

    \noindent Finally, since
    $g\circ
f^{s_{o}}_{1}\simeq f^{g(s_{o})}_{2} \Longrightarrow g\circ
f^{(\widetilde{\overline{g}}\circ g)(s_{o})}_{1}\simeq
f^{g(s_{o})}_{2},$ then similarly, there is  a fiber map $k$
satisfying above  requirement properties for $h$.
 \end{proof}
%%====================================================================================

\begin{cor}\label{a1116}
 If   two   fibrations
$[S_1,f_1,O,F^1_{r_{o}},\Theta_{L_{f_{1}}}]$ and
$[S_2,f_2,O,F^2_{r_{o}},\Theta_{L_{f_{2}}}]$ have conjugate
$Lf-$functions by $g\in
 H(F^{1}_{r_{o}},F^{2}_{r_{o}})$, then Theorem \ref{a1115} holds for
any $s_{o}\in  F^{1}_{r_{o}}$.
\end{cor}
\begin{proof}
It is   clear that if two fibrations $f_{1}$  and $f_{2}$ have
conjugate $Lf-$functions by $g\in H(F^{1}_{r_{o}},F^{2}_{r_{o}})$,
then they have conjugate $Lf-$restrictions $f^{s_{o}}$ and
$f^{g(s_{o})}$ by  $g\in H(F^{1}_{r_{o}},F^{2}_{r_{o}})$, for any
$s_{o}\in F^{1}_{r_{o}}$. Hence  Theorem \ref{a1115} holds for any
$s_{o}\in F^{1}_{r_{o}}$.
 \end{proof}
%%====================================================================================

\noindent We explain  in the following corollary that if $S_{1}$,
$S_{2}$ are simply connected in Theorem \ref{a1115}, then two loop
spaces $   \Omega(S_{1},s_{o}) $ and $
  \Omega(S_{2},g(s_{o})) $ are of the same homotopy type.
%%====================================================================================

\begin{cor}\label{a1117}
 Let  $[S_1,f_1,O,F^1_{r_{o}},\Theta_{L_{f_{1}}}]$ and
$[S_2,f_2,O,F^2_{r_{o}},\Theta_{L_{f_{2}}}]$ be  fibrations with
conjugate $Lf-$restrictions $f^{s_{o}}_{1}$ and $f^{g(s_{o})}_{2}$
by $g\in  H(F^{1}_{r_{o}},F^{2}_{r_{o}})$ and $S_{1}$, $S_{2}$ be
simply connected spaces. Let $  \Omega(O,r_{o}) \simeq  ANR$.  If
$ F^{1}_{r_{o}}$ and $ F^{2}_{r_{o}} $ are pathwise  connected and
ANR's, then \[
  \Omega(S_{1},s_{o}) \simeq   \Omega(S_{2},g(s_{o})).\]
\end{cor}
\begin{proof}
 Since $S_{1}$ and  $S_{2}$ are simply connected, then it is
clear that $   \Omega(S_{1},s_{o}) $ and $  \Omega(S_{2},g(s_{o}))
$ are pathwise  connected. Since $  \Omega(O,r_{o}) \simeq   ANR $
then Theorem \ref{a113} shows that  $   \Omega(S_{1},s_{o}) $ and
$  \Omega(S_{2},g(s_{o})) $ are  dominated by  ANR's. By Theorem
\ref{a1115}, there is a
 map \[h_{o}:
  \Omega(S_{1},s_{o}) \longrightarrow  \Omega(S_{2},g(s_{o})) \]
induces isomorphisms between the homotopy groups. Hence by Theorem
\ref{a112}, $h_{o}$ is homotopy equivalence.
 \end{proof}

%%====================================================================================

\noindent In the next step, we employ Theorem \ref{a1115} to
satisfy FHE relation  for fibrations $\Gamma(f ,s_{o})$. Figure 2
in Theorem \ref{a1115} suggests that perhaps in some sense there
is FHE relation between $\Gamma(f_{1},s_{o})$ and
$\Gamma(f_{2},g(s_{o}))$.  But the notion of the  FHE relation
applied to fibrations having a common base. One might try
 fibering $\Gamma(S_{1},F^{1}_{r_{o}},s_{o})$ over
$ F^{1}_{r_{o}} $ using the  map $g\circ \Psi_{s_{o}}$ but in
general this will not give rise to    fibering since it might
happen that $g(F^{1}_{r_{o}})=g(s_{o})$  and in this  case $g\circ
\Psi_{s_{o}}$ would not be onto if $F^{2}_{r_{o}}$ consisted of
more than one point. Hence we give the restrictions such as
$F^{1}_{r_{o}}=F^{2}_{r_{o}}=F_{r_{o}}$ and $g: F_{r_{o}}
\longrightarrow F_{r_{o}} $ is a homeomorphism  map.\\

%%====================================================================================
\noindent Let $[S,f,O, F_{r_{o}},\Theta_{L_{f}}]$ be      a
fibration and $g$ be a homeomorphism  map of  $O$ onto a
topological semigroup $O'$. Then the composition $g\circ
f:S\longrightarrow O'$ is also a fibration denoted by $[f]_{g}$.

%%====================================================================================
\begin{thm}\label{a1118}
 Let $[S_1,f_1,O,F_{r_{o}},\Theta_{L_{f_{1}}}]$ and
$[S_2,f_2,O,F_{r_{o}},\Theta_{L_{f_{2}}}]$  be
   fibrations with   conjugate
$Lf-$restrictions $f^{s_{o}}_{1}$ and $f^{g(s_{o})}_{2}$ by a
homeomorphism $g\in  H(F_{r_{o}},F_{r_{o}})$, where $s_{o}\in
F_{r_{o}}$, and  $F_{r_{o}}$ be a  pathwise connected ANR. If
$S_{1}$ and $S_{2}$ are simply connected and
 $ \Omega(O,r_{o}) \simeq
ANR$, then $[\Gamma(f_{1},s_{o})]_{g}$ and
$\Gamma(f_{2},g(s_{o}))$ are fiber homotopy equivalent.
\end{thm}
\begin{proof}
 By Theorem \ref{a1115}, in Figure 6, there is a fiber map
\[h:\Gamma(S_{1},F_{r_{o}},s_{o})\longrightarrow
\Gamma(S_{2},F_{r_{o}},g(s_{o})).\] Let
\begin{eqnarray*}
  A_{1} =: (g\circ \Psi_{s_{o}})^{-1}(s_{o})
  &=& \{\alpha\in
 S_{1}^I:\alpha(0)=s_{o},\mbox{ }\alpha(1)=g^{-1}(s_{o})\},
\end{eqnarray*}
and
 \[A_{2}=: \Psi_{g(s_{o})}^{-1}(s_{o})=\{\alpha\in
 S_{2}^I:\alpha(0)=g(s_{o}),\mbox{ }\alpha(1)=s_{o}\}.\]
 \begin{figure}[h!]
   \begin{center}
 \begin{displaymath} \xymatrix{
\Gamma(S_{1},F_{r_{o}},s_{o}) \ar[d]_{h}
           \ar[drrrrr]^{g\circ \Psi_{s_{o}}}&\\
 \Gamma(S_{2},F_{r_{o}},g(s_{o}))\ar[rrrrr]_{\Psi_{g(s_{o})}}   &&&&& F_{r_{o}} }
\end{displaymath}
  \end{center}
   \vspace{-2mm}
 \centerline{Figure 6}
 \label{apic9}
\end{figure}

\noindent We observe that $A_{1}$ and $A_{2}$ are  fiber spaces
for fibrations $[\Gamma(f_{1},s_{o})]_{g}$ and
$\Gamma(f_{2},g(s_{o}))$ over $s_{o}$, respectively. Hence
homotopy sequences for the  fibrations $[\Gamma(f_{1},s_{o})]_{g}$
and $\Gamma(f_{2},g(s_{o}))$ in Theorem \ref{a1115} show that
$h|_{A_{1}}:A_{1}\longrightarrow A_{2}$ induces isomorphisms
between $\pi_{n}(A_{1})$ and $\pi_{n}(A_{2})$ for a positive
integer $n>0$.\\

\noindent Now since $S_{1}$ and  $S_{2}$ are simply connected
spaces, then $
  \Omega(S_{1},s_{o}) $ and $  \Omega(S_{2},g(s_{o})) $ are
pathwise  connected. Since $  \Omega(O,r_{o}) \simeq
 ANR $,  then
 by Theorem \ref{a113}, $   \Omega(S_{1},s_{o}) $ and
$  \Omega(S_{2},g(s_{o})) $ are  dominated by
 ANR's.  Also these loop spaces are  fiber
spaces for the  fibrations $\Gamma(f_{1},s_{o})$ and
$\Gamma(f_{2},g(s_{o}))$  over $s_{o}$, respectively. Since $
F_{r_{o}} $ is a pathwise  connected and  by Theorem \ref{a114},
all fiber spaces are of the same  homotopy type, then $A_{1}$ and
$A_{2}$ are pathwise  connected and dominated by
 ANR's.  Since
$h|_{A_{1}}:A_{1}\longrightarrow A_{2}$ induces isomorphisms
between $\pi_{n}(A_{1})$ and $\pi_{n}(A_{2})$ for a positive
integer $n>0$,  then  by Theorem \ref{a112},
$h|_{A_{1}}:A_{1}\longrightarrow A_{2}$ is a homotopy equivalence.
 Therefore since $F_{r_{o}}$ is pathwise
 connected ANR, then by    Fadell-Dold
theorem, we get that  $[\Gamma(f_{1},s_{o})]_{g}$ and
$\Gamma(f_{2},g(s_{o}))$ are fiber homotopy equivalent.
 \end{proof}
%%====================================================================================

\begin{cor}\label{a1119}
 Let $[S,f,O, F_{r_{o}},\Theta_{L_{f}}]$  be    a   fibration
with simply connected ANR fiber space $F_{r_{o}}$ and with simply
connected base $O$ such that $  \Omega(O,r_{o}) \simeq
 ANR$. If there is a map
$k:O\longrightarrow S$ such that $f\circ k=id_{O}$, then
$\Gamma(f,k(r_{o}))$ and $\Gamma(P_{1},(r_{o},k(r_{o})))$ are
fiber homotopy equivalent,   where $P_{1}: O\times F_{r_{o}}
\longrightarrow O $ is the first   fibration.
\end{cor}
\begin{proof}
 Since  $k(r_{o})\in S$ and   $f\circ k=id_{O}$, then
\begin{eqnarray*}
  f^{k(r_{o})}(\alpha) = \Theta_{L_{f}}(\alpha,k(r_{o})) &=& L_{f}(k(r_{o}),\alpha)(1)\\
   &=& L_{f}[k(r_{o}),f\circ(k\circ\alpha)](1)
 \end{eqnarray*}
 for all $\alpha\in  \Omega(O,r_{o})$. We observe  easily   that  the $Lf-$restriction $f^{k(r_{o})}$
is  homotopic to the  map $L:  \Omega(O,r_{o}) \longrightarrow
 F_{r_{o}} $ which is defined by \[
 \Omega(\alpha)=(k\circ\alpha)(1)=k(r_{o}) \quad\mbox{for}\mbox{
}\alpha\in  \Omega(O,r_{o}) \] by using the form of homotopy $H$
in the proof of Theorem \ref{a1115}.  Consider $ F_{r_{o}} $ as
fiber space
   for $P_{1}$ because there is a homeomorphism  map between
  it and $P^{-1}_{1}(r_{o})=\{r_{o}\}\times F_{r_{o}}$. Now from Example \ref{a1113}, we get
  that  fibration $P_{1}$ has
$Lf-$restriction $P^{(r_{o},k(r_{o}))}_{1}:  \Omega(O,r_{o})
\longrightarrow
 F_{r_{o}} $ given by
\[P^{(r_{o},k(r_{o}))}_{1}(\alpha)=k(r_{o})\quad\mbox{for
}\mbox{ }\alpha\in  \Omega(O,r_{o}).\] Hence
  $f^{k(r_{o})}\simeq L=P^{k(r_{o})}_{1}$, that is,   fibrations
  $f$ and $P_{1}$ have conjugate $Lf-$restrictions  $f^{k(r_{o})}$
  and $P^{(r_{o},k(r_{o}))}_{1}$ by a homeomorphism
  $g=id_{F_{r_{o}}}\in  H(F_{r_{o}},F_{r_{o}})$. Hence by
  theorem above,
  $\Gamma(f,k(r_{o}))$ and $\Gamma(P_{1},(r_{o},k(r_{o})))$ are
fiber homotopy equivalent.
 \end{proof}

  %%%###################################################################################################3

\section{Fibration $\Sigma(f)$ and $Lf-$function}

\noindent Here, we will introduce the role of homotopy sequences
of fibrations (using $Lf-$function ) in satisfying
  FHE between two fibrations $\Sigma(f_{1})$ and $\Sigma(f_{2})$which are induced by two
fibrations $[S_1,f_1,O,F^1_{r_{o}},\Theta_{L_{f_{1}}}]$ and
$[S_2,f_2,O,F^2_{r_{o}},\Theta_{L_{f_{2}}}]$ over a common base
$O$.

%%====================================================================================

\noindent In the following theorem, we show that for two
  fibrations $f_{1}$  and $f_{2}$   with  conjugate
$Lf-$functions, there are two fiber maps
 between  two   fibrations $\Sigma(f_{1})$  and
$\Gamma(f_{2})$.

%%====================================================================================
\begin{thm}\label{a1120}
 Let $[S_1,f_1,O,F^1_{r_{o}},\Theta_{L_{f_{1}}}]$ and
$[S_2,f_2,O,F^2_{r_{o}},\Theta_{L_{f_{2}}}]$ be   fibrations with
  conjugate $Lf-$functions by $g\in
 H(F^{1}_{r_{o}},F^{2}_{r_{o}})$. Then there are two fiber maps
\[D:\Sigma(S_{1},F^{1}_{r_{o}})\longrightarrow
\Sigma(S_{2},F^{2}_{r_{o}})\quad \mbox{and}\quad
R:\Sigma(S_{2},F^{2}_{r_{o}})\longrightarrow
\Sigma(S_{1},F^{1}_{r_{o}})\] over $g\times g$ and
$\widetilde{\overline{g}}\times \widetilde{\overline{g}}$,
respectively. That is,   Figure 7 is a commutative
\begin{figure}[h!]
   \begin{center}
 \begin{displaymath} \xymatrix{
 \Sigma(S_{1},F^{1}_{r_{o}}) \ar[rr]^D \ar[d]^{\Phi_{1}}&&
           \Sigma(S_{2},F^{2}_{r_{o}})\ar[rr]^R\ar[d]^{\Phi_{2}}&& \Sigma(S_{1},F^{1}_{r_{o}})\ar[d]^{\Phi_{1}} &    \\
  F^{1}_{r_{o}}\times F^{1}_{r_{o}} \ar[rr]_{g\times g}   &&  F^{2}_{r_{o}}\times F^{2}_{r_{o}} \ar[rr]_{\widetilde{\overline{g}}\times \widetilde{\overline{g}}} &&
  F^{1}_{r_{o}}\times F^{1}_{r_{o}} }
\end{displaymath}
  \end{center}
 \vspace{-2mm}
 \centerline{Figure 7}
 \label{apic10}
\end{figure}
\end{thm}
\begin{proof}
 Firstly, we will define  fiber map $D$. By the hypothesis we get that
\[\Theta_{L_{f_{1}}}\simeq \widetilde{\overline{g}}\circ
\Theta_{L_{f_{2}}}\circ(id_{ \Omega(O,r_{o})}\times g).\] This
implies
\[g\circ \Theta_{L_{f_{1}}}\simeq
\Theta_{L_{f_{2}}}\circ(id_{ \Omega(O,r_{o})}\times g).\]Hence
there is
   a homotopy $T:  \Omega(O,r_{o})\times
F^{1}_{r_{o}} \longrightarrow (F^{2}_{r_{o}})^I$ such that
\begin{eqnarray*}
  T(\alpha,s)(0) &=& [\Theta_{L_{f_{2}}}\circ(id_{ \Omega(O,r_{o})}\times
g)](\alpha,s) \\
   &=& \Theta_{L_{f_{2}}}(\alpha,g(s))
\end{eqnarray*}  and
\begin{eqnarray*}
  T(\alpha,s)(1) &=& [g\circ\Theta_{L_{f_{1}}}](\alpha,s) \\
   &=& g[\Theta_{L_{f_{1}}}(\alpha,s)]
\end{eqnarray*}
for all $\alpha\in  \Omega(O,r_{o}),s\in F^{1}_{r_{o}}$.
  Define      a map
$L''_{f_{2}}:\Sigma(S_{1},F^{1}_{r_{o}})\longrightarrow
\Sigma(S_{2},F^{2}_{r_{o}})$ by
\[L''_{f_{2}}(\alpha)=L_{f_{2}}(g(\alpha(0)),f_{1}\circ \alpha) \quad \mbox{for} \mbox{ }
\alpha\in \Sigma(S_{1},F^{1}_{r_{o}}),\]and  for $\alpha\in
\Sigma(S_{1},F^{1}_{r_{o}})$, we can use the homotopy $H$ (which
is defined in the proof of Theorem \ref{a1115}) to define the path
$W(\alpha)\in (F^{2}_{r_{o}})^I$  by
\[W(\alpha)(t)=g\{[H(\alpha,t)](1)\} \quad \mbox{for} \mbox{ }
 t\in I.\] Now we can define
a map $D:\Sigma(S_{1},F^{1}_{r_{o}})\longrightarrow
\Sigma(S_{2},F^{2}_{r_{o}})$  by
\begin{eqnarray}\label{eeq5}
  D(\alpha)&=&[L''_{f_{2}}(\alpha)\star T(f_{1}\circ \alpha)]\star W(\alpha) \quad \mbox{for} \mbox{ }
\alpha\in \Sigma(S_{1},F^{1}_{r_{o}}).
\end{eqnarray}
Hence it is clear that $D$ is well defined as a continuous map. We
get that
\begin{eqnarray*}
 [\Phi_{2}\circ D](\alpha)=[D(\alpha)(0),D(\alpha)(1)]&=&[L''_{f_{2}}(\alpha)(0),W(\alpha)(1)] \\
   &=& [g(\alpha(0)),g(\alpha(1))] \\
   &=&(g\times g)(\alpha(0),\alpha(1)) \\
   &=& [(g\times g)\times \Phi_{1}](\alpha)
\end{eqnarray*}
 for all $ \alpha\in
\Sigma(S_{1},F^{1}_{r_{o}})$. That is, $D$ is      a fiber map
over $g\times g$. Secondly, we can find      a fiber map  $R$ by
above similar manner.
 \end{proof}
%%====================================================================================

\noindent In the proof of Theorem \ref{a1120}, the two  fiber maps
have properties:
\[D|_{ \Omega(S_{1},s_{o})}=h_{o}\quad\mbox{and}\quad R|_{ \Omega(S_{2},g(s_{o}))}=k_{o},\] where
$h_{o}$ and $k_{o}$ are defined in Theorem \ref{a1115} and
$s_{o}\in  F^{1}_{r_{o}}$.  In proof of Theorem \ref{a1115}, it is
clear that the  map $G$ is a restriction of  a map $T$ on $
 \Omega(O,r_{o}) $, the  map $L'_{f_{2}}$ is a restriction of a
 map $L''_{f_{2}}$ on $\Gamma(S_{1},F^{1}_{r_{o}},s_{o})$, and
the  map $M'$ is a restriction of    a map $W$ on
$\Gamma(S_{1},F^{1}_{r_{o}},s_{o})$. Hence from Equations
\ref{eeq4} and \ref{eeq5}, we get that the  map $h$ is a
restriction of    a map $D$ on
$\Gamma(S_{1},F^{1}_{r_{o}},s_{o})$, that is,$D|_{
\Omega(S_{1},s_{o})}=h_{o}$. Similarly, for
$R|_{ \Omega(S_{2},g(s_{o}))}=k_{o}$.\\

%%====================================================================================

\noindent Also we introduce   theorem  about the functor $\Sigma$
which is similar of   Theorem  \ref{a1118}.
%%====================================================================================
\begin{thm}\label{a1121}
 Let $[S_1,f_1,O,F_{r_{o}},\Theta_{L_{f_{1}}}]$ and
$[S_2,f_2,O,F_{r_{o}},\Theta_{L_{f_{2}}}]$ be  fibrations with
conjugate $Lf-$functions by a homeomorphism $g\in
 H(F_{r_{o}},F_{r_{o}})$, where $s_{o}\in F_{r_{o}}$, and $F_{r_{o}}$ be a
common pathwise  connected ANR. If $S_{1}$, $S_{2}$ are simply
connected and
 $  \Omega(O,r_{o}) \simeq
 ANR$, then
$[\Sigma(f_{1})]_{g\times g}$ and $\Sigma(f_{2})$ are fiber
homotopy equivalent.
\end{thm}
\begin{proof}
 By Theorem above, there is      a fiber map
$D:\Sigma(S_{1},F_{r_{o}})\longrightarrow \Sigma(S_{2},F_{r_{o}})$
in Figure 8.
\begin{figure}[h!]
   \begin{center}
 \begin{displaymath} \xymatrix{
\Sigma(S_{1},F_{r_{o}}) \ar[d]_{D}
           \ar[drrrrr]^{(g\times g)\circ \Phi_{1}}&\\
 \Sigma(S_{2},F_{r_{o}})\ar[rrrrr]_{\Phi_{2}}   &&&&& F_{r_{o}}\times F_{r_{o}} }
\end{displaymath}
  \end{center}
  \vspace{-2mm}
 \centerline{Figure 8}
 \label{apic11}
\end{figure}
\newline Let $B_{1} = [(g\times g)\circ \Phi_{1}]^{-1}(s_{o},s_{o})$ and $B_{2}= \Phi^{-1}_{2}(s_{o},s_{o}))$, then
\begin{eqnarray*}
  B_{1} &=& [(g\times g)\circ \Phi_{1}]^{-1}(s_{o},s_{o}) \\
   &=& \Phi^{-1}_{1}[(g^{-1}\times g^{-1})(s_{o},s_{o})]\\
     &=& \{\alpha\in
 S_{1}^I:\alpha(0)=g^{-1}(s_{o}),\mbox{
}\alpha(1)=g^{-1}(s_{o})\}\\
&=& \Omega(S_{1},g^{-1}(s_{o})),
\end{eqnarray*}
and
 \[B_{2}= \Phi^{-1}_{2}(s_{o},s_{o}))=\{\alpha\in
 S_{2}^I:\alpha(0)=s_{o},\mbox{
}\alpha(1)=s_{o}\}= \Omega(S_{2},s_{o}).\] We observe that $B_{1}$
and $B_{2}$ are fiber spaces for  two  fibrations
$[\Sigma(f_{1})]_{g\times g}$ and $\Sigma(f_{2})$  over
$(s_{o},s_{o})$, respectively. Hence  homotopy  sequences for  two
 fibrations $[\Sigma(f_{1})]_{g\times g}$ and
$\Sigma(f_{2})$ in Theorem \ref{a1115} show  that
$D|_{B_{1}}:B_{1}\longrightarrow B_{2}$ induces isomorphisms
between $\pi_{n}(B_{1})$ and $\pi_{n}(B_{2})$ for a positive
integer $n>0$. Since $S_{1}$ and  $S_{2}$ are simply connected
spaces, then $B_{1}$ and $B_{2}$ are pathwise
$S_{\mathcal{N}_{i}}-$connected.
 Since
$ \Omega(O,r_{o}) \simeq   ANR$,   then  by Theorem \ref{a113}, we
get that $B_{1}$ and $B_{2}$ are dominated by ANR's.  And since
$D|_{B_{1}}:B_{1}\longrightarrow B_{2}$ induces isomorphisms
between $\pi_{n}(B_{1})$ and $\pi_{n}(B_{2})$ for a positive
integer $n>0$, then  by Theorem \ref{a112},
$D|_{B_{1}}:B_{1}\longrightarrow B_{2}$ is a homotopy equivalence.
Hence since $F_{r_{o}}\times F_{r_{o}} $ is pathwise
 connected ANR, then by  Fadell-Dold theorem,
we get that $[\Sigma(f_{1})]_{g\times g}$ and $\Sigma(f_{2})$ are
fiber homotopy equivalent.
 \end{proof}

\noindent {\bf Conclusion:} \noindent Further we also prove some
theorems related to fiber homotopy  equivalent classes by using
the fiber homotopy sequences of homotopy groups. Thus we show the
role of these fiber homotopy sequences in order to get the
required fiber map in Fadell-Dold theorem. Further, the possible
practical use of our theorems as applications will provide some
solutions for the classification problem in Hurewicz fibration
theory by
using Fadell-Dold theorem.\\

%\noindent {\bf Acknowledgement} The authors are grateful to Paul G
%Goerss for pointing out the references \cite{dwyer1},
%\cite{dwyer2} and \cite{dwyer4}.

%%%%%%%%%%%%%%%%%%%%%%%%%%%%%%%%%%%%%%%%%%%%%%%%%%%%%%%%%%%%%%%%%%%%%%%%%%%%%%%%%%%%%%%%%%%%%%%%%%%%%%%%%%%%

\bibliographystyle{plain}
%\bibliography{Adomian8}

\end{document}